\documentclass[oneside]{amsart}
\usepackage{amsmath, amssymb}
\usepackage{amsthm}
\usepackage{graphicx}
\newtheorem{theorem}{Theorem}[section]
\newtheorem{proposition}[theorem]{Proposition}
\newtheorem{lemma}[theorem]{Lemma}
\newtheorem{corollary}[theorem]{Corollary}

\newtheorem{example}{Example}
\newtheorem{question}{Question}

\def\F{\mathcal{F} }

\def\S{\mathbb{S} } 
\def\T{\mathbb{T} } 

\def\F{\mathcal{F} }

\def\R{\mathbb{R} } 
\def\Z{\mathbb{Z} } 
\def\nbd{neighborhood } 
 
\def\R{\mathbb{R} }

\title[Recurrent and Non-wandering properties for foliations]
{Recurrent and Non-wandering properties for foliations} 
\author{Tomoo Yokoyama}
\date{\today}

\address{Department of Mathematics, Kyoto University of Education/JST PRESTO, 
1 Fujinomori, Fukakusa, Fushimi-ku Kyoto, 612-8522, Japan \\
}
\email{tomoo@kyokyo-u.ac.jp}

\thanks{The author is partially supported
by the JST PRESTO  Grant Number JPMJPR16E 
and by Grant for Basic Science Research Projects from The Sumitomo Foundation at Department of Mathematics, Kyoto University of Education.}

\begin{document}

\maketitle

\begin{abstract}
In this paper, 
we define the recurrence and ``non-wandering'' for decompositions. 
The following inclusion relations hold for codimension one foliations 
on closed $3$-manifolds: 
$\{$minimal$\} \sqcup \{$compact$\}$ 
$\subsetneq$ 
$\{$pointwise almost periodic$\}$   
$\subsetneq$ 
$\{$recurrent$\}$ 
$\subsetneq$ 
$\{$non-wandering$\}$ 
$\subsetneq$ 
$\{$Reebless$\}$. 
A non-wandering codimension one  $C^2$ foliation on a closed connected $3$-manifold 
which has no leaf with uncountably many ends 
is minimal (resp. compact) if and only if it has no compact (resp. locally dense) leaves. 
In addition, 
the fundamental groups of all leaves of 
a codimension one transversely orientable $C^2$ foliation $\F$ on a closed $3$-manifold 
have the same polynomial growth
if and only if 
$\F$ is without holonomy and has a leaf whose fundamental group has polynomial growth. 
\end{abstract}

\section{Introduction and preliminaries}

In 1927, 
Birkhoff introduced the concepts of non-wandering points and 
recurrent points \cite{Bi}. 
Using these concepts, 
we can describe and capture 
sustained or stationary dynamical behaviors 
and 
conservative dynamics. 
In \cite{Th}, 
it has shown that 
a closed manifold $M$ has a smooth codimension one foliation 
if and only if the Euler characteristic of $M$ is zero. 
In particular, 
each three dimensional closed manifold has 
codimension one smooth foliation. 
In this paper, 
we define the recurrence and ``non-wandering'' for decompositions.  
As usual dynamical systems, 
the following relations hold for a decomposition:
\begin{center}
pointwise almost periodic 
$\subsetneq$ 
recurrent 
$\subsetneq$ 
non-wandering.  
\end{center} 
In particular, the inclusions hold for 
codimension one  foliations  on closed $3$-manifolds.  
Moreover, 
let $\F$ be a codimension one foliation on a closed $3$-manifold $M$.   
If $\F$ is non-wandering, then 
$\F$ is $\pi_1$-injective and so Reebless.  
Therefore 
there are no codimension one non-wandering foliations on some closed $3$-manifolds.  
%
%
%
For a codimension one transversely orientable $C^2$ foliation $\F$ on a closed $3$-manifold, 
the fundamental groups of all leaves of $\F$ have the same polynomial growth
if and only if 
$\F$ is without holonomy and has a leaf whose fundamental group has polynomial growth. 


\subsection{Topological notions}

A point $x$ of a topological space $(X, \tau)$ is said to be 
$T_1$ if 
the singleton $\{ x \}$ is closed, 
$T_0$ if 
for any points $y \neq x \in X$, 
there is no open subset $U$ of $X$ 
such that 
$\{x, y \} \cap U$ is a singleton, 
$T_D$ \cite{AT} if $\overline{\{ x \}} - \{ x \}$ is closed, 
($\tau$-)recurrent if 
$\{ x \}$ is $T_1$ or non-$T_D$, 
and   
($\tau$-)wandering if 
there is an open \nbd of $\{ x \}$ which consists of non-recurrent (i.e. non-$T_1$ $T_D$) elements, 
and   
($\tau$-)non-wandering if 
it is not ($\tau$-)wandering 
(i.e. there is no open \nbd of $\{ x \}$ which consists of non-recurrent (i.e. non-$T_1$ $T_D$) elements). 
Note that 
the set of recurrent points  of a flow on a complete metric space 
is dense in 
the set of non-wandering points 
(Theorem III.2.12, III.2.15 \cite{BS}). 
%
For a point $x$ of a topological space $(X, \tau)$, 
define the (point) class $\hat{x}$ by 
$\hat{x} : = \{ y \in X \mid \overline{\{ y \}} = \overline{\{ x \}} \}$.  
The quotient space by the classes is called the class space 
and denoted by $\hat{X}$. 
A point $x \in X$ is $S_1$ if 
the class $\hat{x}$ is $T_1$ with respect to $\hat{X}$.  
A topological space is $T_0$ (resp. $T_D$, $S_1$) if 
so is each point. 
Note that the $T_1$ (resp. recurrent) property implies the $T_D$ (resp. non-wandering) property. 
%
By the definition of properness, we have that 
a foliation $\F$ on a manifold $M$ is proper 
if and only if 
the leaf space $M/\F$ is $T_D$ with respect to the quotient topology $\tau_{\F}$. 
Moreover, a leaf is proper if and only if it is $T_D$ with respect to $\tau_{\F}$. 

\subsection{Decompositions}

By a decomposition, 
we mean a family $\mathcal{F} $ of pairwise disjoint nonempty subsets of a topological space $(X, \tau)$ 
such that $X = \sqcup \mathcal{F}$, 
where $\sqcup$ is the disjoint union symbol. 
For $L \in \F$, 
we call that $L$ is a proper element if 
$L$ is $T_D$ with respect to $\tau_{\F}$, 
and 
a recurrent element (resp. a wandering element, a non-wandering element) if 
it is $\tau_{\F}$-recurrent (resp. $\tau_{\F}$-wandering, $\tau_{\F}$-non-wandering). 
A subset of $X$ is saturated if it is a union of elements of $\F$. 
Denote by $\mathrm{Cl}$ (resp. $\mathrm{P}$, $\mathrm{R}$)
the set of closed elements  
(resp. proper non-closed elements, non-closed recurrent elements). 
Then 
$\mathrm{P}$ is the complement of the union of recurrent elements,  
$\mathrm{R}$ is the complement of the set of the union of proper elements, 
and 
$X = \mathrm{Cl} \sqcup \mathrm{P} \sqcup \mathrm{R}$,  
where $\sqcup$ is the disjoint union symbol. 
%
Note that 
i) $L$ is a proper element  
if and only if 
$\overline{L} -L$ is closed; 
ii) $L$ is a recurrent  element  if and only if it is either closed or non-proper; 
iii) $L$ is a wandering element  if and only if 
$L \subseteq \mathrm{int}\mathrm{P}$; 
and 
iv) $L$ is a non-wandering element if and only if 
there is no open saturated \nbd of $L$ which consists of non-recurrent (i.e. non-closed proper) elements. 
%

\subsection{Notions of dynamical systems}
A decomposition $\F$ on a topological space $X$ is recurrent (resp. non-wandering) 
if 
so is each element. 
For any $x \in X$, 
denote by $L_x$ the element of $\mathcal{F} $ containing $x$. 
Recall that 
$\mathcal{F}$ is pointwise almost periodic  if 
the set $\hat{\F}$ of all closures of elements of $\mathcal{F} $ also is a decomposition. 
Note 
$\F$ is pointwise almost periodic 
if and only if 
the class space $X/\hat{\F}$ is $T_1$ 
(i.e. $X/\F$ is $S_1$), 
where 
$X/\hat{\F}$ is a quotient space $X/\sim$ 
defined as follows: 
$x \sim y$ if 
$\overline{L_x} = \overline{L_y}$. 
A decomposition  
$\mathcal{F}$ is $R$-closed if 
$R := \{ (x, y) \mid y \in \overline{L_x} \}$ is closed. 
By Corollary 1.4 \cite{Y}, 
$R$-closedness implies pointwise almost periodicity. 
By Lemma 2.2  \cite{Y}, 
a pointwise almost periodic decomposition $\F$ of a compact Hausdorff space $X$ 
is $R$-closed 
if and only if 
the class space $X/\hat{\F}$ is Hausdorff.

\subsection{Reeb components}

By a foliation, we mean a continuous foliation. 
A Reeb component of codimension one foliation on a $3$-manifold (resp. surface) is 
a solid torus whose boundary is a toral compact leaf and 
which consists of one toral leaf and planer leaves 
(resp. a closed annulus whose boundary consists of two circular compact leaves 
and which consists of two circular leaves and non-compact leaves as Figure \ref{Reeb}). 
A codimension one foliation on a $3$-manifold (resp. surface) is Reebless if 
there are no Reeb components. 

\begin{figure}
\begin{center}
\includegraphics[scale=0.43]{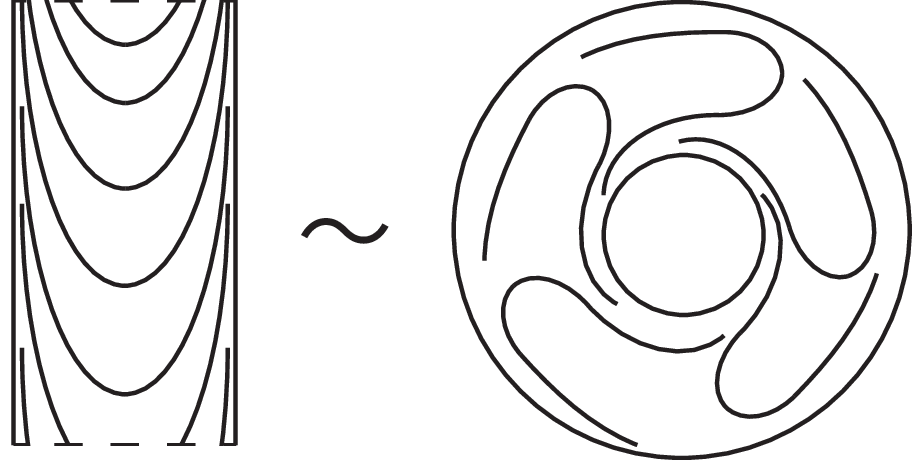}
\end{center} 
\caption{Two-dimensional Reeb component}\label{Reeb}
\end{figure}

\subsection{Vanishing cycle and $\pi_1$-injectivity}

A loop  $\gamma $ on a leaf $L$ of a foliation $\F$ on a manifold $M$ 
is called a vanishing cycle in the sense of Novikov 
if there is a mapping $F: S^1 \times [0,1] \to M$ such that 
arcs $F( x, [0,1])$ for every $x \in S^1$ are transverse to $\F$, 
each loop $F( S^1, t)$ for any $t \in [0,1]$ is contained in some leaf 
$L_t$,   
$[F( S^1 ,0)] = [\gamma ] \neq 1 \in \pi_1(L)$ and 
$[F( S^1 ,t)] = [\gamma] = 1 \in \pi_1(L_t)$ for all $t \in (0,1]$. 
A leaf $L$ is $\pi_1$-injective 
if 
the inclusion mapping $i : L \to M$ 
induces the injection $\pi_1(i): \pi_1(L) \to \pi_1(M)$ 
between the fundamental groups. 

\section{Properties of decompositions}

Let $\mathcal{F} $ be a decomposition of a topological space $X$. 
We show the following inclusion relations.

\begin{lemma}
$
\{ S_1\text{ point} \} \subsetneq 
\{ \text{recurrent point} \} \subsetneq 
\{ \text{non-wandering point} \} 
$
%
\end{lemma}

\begin{proof}
By definition, 
recurrence implies non-wandering property. 
Fix any $S_1$ point $y$ of a topological space. 
Since a closed point is recurrent, 
we may assume that $y$ is not closed. 
This implies that 
there is an element $x \in \overline{y} - \{ y \}$. 
Since $y$ is $S_1$, 
we obtain 
$\overline{\{x \}} = \overline{\{ y \}}$. 
Then $\overline{\{x \}} = \overline{\{ y \}} \supsetneq 
\overline{\{ y \}} - \{ y \} = \overline{\{x \}} - \{ y \} \supsetneq \{x \}$. 
Thus $\overline{\{ y \}} = \overline{\overline{\{ y \}} - \{ y \}}$ 
and so 
$\overline{\{ y \}} - \{ y \}$ is not closed. 
This shows that 
$y$ is non-$T_D$ and so recurrent. 
\end{proof}

The previous lemma can be interpreted as follows. 

\begin{proposition}\label{lem11}
A pointwise almost periodic decomposition is recurrent. 
\end{proposition}

\begin{proposition}\label{lem11}
A recurrent decomposition is non-wandering. 
\end{proposition}


In general, 
the converses of these inclusion relations are not true 
even if decompositions are codimension one foliations 
(see Example \ref{prop2.3} and \ref{prop2.4}). 
Notice that the Denjoy foliation on $\T^2$ is wandering but Reebless. 
We state a characterization of non-wandering property.

\begin{lemma}
An element $L$ of a decomposition $\F$ on a topological space $X$ is non-wandering point if and only if 
$L \not\subseteq \mathrm{int}\mathrm{P}$. 
\end{lemma}

\begin{proof}
Note that the set of recurrent elements is 
$\mathrm{Cl} \sqcup \mathrm{R} = X - \mathrm{P}$. 
Suppose that $L$ is wandering. 
Then there is an open saturated \nbd $U \subseteq \mathrm{P}$ of $L$ 
and so $L \subseteq \mathrm{int}\mathrm{P}$.  
Conversely, 
suppose that $L \subseteq \mathrm{int}\mathrm{P}$. 
Since $\mathrm{int}\mathrm{P}$ is an open saturated \nbd of $L$ 
which contains no recurrent elements, 
the element $L$ is wandering. 
\end{proof}

\begin{corollary}
A decomposition $\F$ on a topological space $X$ is non-wandering if and only if 
$\mathrm{int}\mathrm{P} = \emptyset$. 
\end{corollary}

\begin{proof}
If $\mathrm{int}\mathrm{P} \neq \emptyset$, then 
each element in $\mathrm{int}\mathrm{P}$ is wandering. 
Conversely, if $\F$ is wandering, then there is a wandering element $L$ contained in $\mathrm{int}\mathrm{P}$ 
and so $\mathrm{int}\mathrm{P} \neq \emptyset$. 
\end{proof}

\subsection{Properties of codimension one foliations}

For a codimension one foliation $\F$ on a manifold $M$, 
denote by 
$\mathrm{LD}$ 
the union of locally dense leaves 
and by 
$\mathrm{E}$ 
the union of exceptional leaves.  
Recall that $\mathrm{R}$ is the union of non-proper leaves. 
Then $\mathrm{R} = \mathrm{LD} \sqcup \mathrm{E}$ 
and so 
$M = \mathrm{Cl} \sqcup \mathrm{P} \sqcup \mathrm{LD} \sqcup \mathrm{E}$.   
We characterize non-wandering property. 

\begin{lemma}\label{lem13}
A codimension one foliation $\F$ on a closed manifold $M$ is non-wandering 
if and only if  
$\overline{\mathrm{R}} \supseteq M - \mathrm{Cl}$. 
\end{lemma}

\begin{proof}
Suppose that 
$\F$ is non-wandering. 
By definitions, 
we have 
$M = \mathrm{Cl} \sqcup \mathrm{P} \sqcup \mathrm{R}$. 
By Theorem4.1.3.V \cite{HH}, 
we have that 
$\mathrm{Cl}$ is closed 
and so that 
$
M - \overline{\mathrm{Cl} \sqcup \mathrm{R}} = 
M - (\mathrm{Cl} \cup \overline{\mathrm{R}}) = 
\mathrm{P}  \setminus \overline{\mathrm{R}}
$
is open.
%
Then  
$\mathrm{P} \setminus \overline{\mathrm{R}} \subset \mathrm{int} \mathrm{P} = \emptyset$ 
and so 
$M - \mathrm{Cl} = \mathrm{P} \sqcup \mathrm{R} \subseteq \overline{\mathrm{R}}$. 
Conversely, 
suppose that 
$\overline{\mathrm{R}} \supseteq M - \mathrm{Cl}$. 
Since $\mathrm{P} \sqcup \mathrm{R} = M - \mathrm{Cl} \subseteq 
\overline{\mathrm{R}}$, 
we have $\mathrm{P} \subset  \overline{\mathrm{R}} - \mathrm{R}$ and so 
$\mathrm{int} \mathrm{P} = \emptyset$. 
This means that 
$\F$ is non-wandering. 
\end{proof}
%
%
%
%
%

%
Recall the following equivalence relation. 

\begin{lemma}[Theorem 5.2 \cite{Y2}]\label{lem12}
Let $F$ be a codimension one foliation on a closed manifold $M$. 
Then the following are equivalent: 
\\
1) 
$\F$ is pointwise almost periodic. 
\\
2)
$\F$ is $R$-closed. 
\\
3) 
$\F$ is minimal or compact. 
\end{lemma}

%
Note that there is 
a pointwise almost periodic codimension two real-analytic foliation on $\T^3$ which is not $R$-closed. 
Indeed, 
consider a diffeomorphism $f : (\R/\Z)^2 \to (\R/\Z)^2$ defined by $f (x, y) := (x + \sin (2 \pi y ), y )$. 
Then the suspension foliation of $f$ on $\T^3$ is desired. 
Moreover, there are compact codimension $q > 2$ foliations on compact manifolds 
which are not $R$-closed \cite{Su,EV,V3}. 
We also obtain the following statement. 

\begin{theorem}
Let $\F$ be a transversely orientable codimension one $C^2$ foliation 
on a closed manifold $M$. 
The following statement are equivalent: 

1) 
$\F$ is without holonomy 
and has a leaf whose fundamental group has polynomial growth. 

2) 
$\F$ is a pointwise almost periodic foliation 
without vanishing cycles 
such that 
the fundamental groups of all leaves have 
the same polynomial growth. 
\end{theorem}

\begin{proof}
Suppose that $\F$ is without holonomy and 
and has a leaf whose fundamental group has polynomial growth. 
By Theorem 6 \cite{Sa}, 
it is topologically conjugate to a foliation defined by a closed one-form. 
In particular, all leaves are diffeomorphic to each other.   
By Theorem 1 \cite{T}, 
the foliation $\F$ is either minimal or compact. 
By Theorem 5.1 \cite{N}, 
each leaf is $\pi_1$-injective. 
This implies the non-existence of vanishing cycles.  
Conversely, 
suppose that 
$\F$ is a pointwise almost periodic foliation 
without vanishing cycles such that 
the fundamental groups of all leaves have 
the same polynomial growth. 
By Lemma \ref{lem12}, 
the foliation $\F$ is either compact or minimal. 
If $\F$ is compact, 
then the transverse orientability 
implies the triviality of holonomy. 
Thus we may assume that $\F$ is minimal. 
Since the union of the leaves without holonomy is dense G$_{\delta}$ \cite{EMT}, 
there is a leaf $ L$ of $\mathcal{F}$ without holonomy. 
Proposition 5.1 \cite{Y4} implies the triviality of holonomy. 
\end{proof}

%

The $C^2$ condition  
and the growth condition are necessary, 
because the Denjoy foliation on $\T^2$ is a non-$R$-closed foliation without holonomy 
and because 
there is a transversely orientable codimension one real-analytic minimal foliation 
with non-trivial holonomy 
such that each leaf has a fundamental group with exponential growth (e.g. Example 3.2 \cite{YT}). 
Note that 
the condition ``the fundamental groups of all leaves have the same polynomial growth'' 
in the previous theorem is necessary, 
because there is a codimension one minimal foliation on a closed $3$-manifold with non-trivial holonomy 
each of whose leaves is either toral or planar. 
In fact, the weakly (un)stable foliation of a transitive Anosov flow on a closed $3$-manifold is desired. 
Moreover, the polynomial growth condition is necessary,  
because there is a codimension one real-analytic minimal foliation on a closed $3$-manifold 
with non-trivial holonomy such that 
all leaves are diffeomorphic to each other (e.g. Example 3.2 \cite{YT}). 
The author would like to know whether 
the pointwise almost periodicity can be replaced in the previous theorem with the non-wandering property. 
In other words, one would like to know a following question. 

\begin{question}\label{q01}
Is there a non-wandering codimension one foliation on a closed manifold 
without vanishing cycles which is not pointwise almost periodic  
such that the fundamental groups of all leaves have the same polynomial growth? 
\end{question}

\section{Codimension one  non-wandering foliations on $3$-manifolds}

\subsection{Properties of non-wandering foliations}

%
The properties of foliations on $3$-manifolds are different from 
those on surfaces. 
For instance, 
in \cite{H}, 
the author has constructed 
a codimension one continuous non-wandering foliation on a closed $3$-manifold $M$
such that 
$\overline{\mathrm{LD}} = \overline{\mathrm{E}} = M = \mathrm{LD} \sqcup \mathrm{E}$ 
(resp. $\mathrm{E} = M$). 
%
This shows that 
$\mathrm{LD}$ (resp. $\mathrm{\mathrm{E}} \sqcup \mathrm{Cl}$) is neither open nor closed in general. 
%
Now
we state 
some properties 
of non-wandering codimension one 
foliations. 

\begin{lemma}\label{lem:25}
A codimension one non-wandering foliation on a connected closed $3$-manifold $M$ 
is $\pi_1$-injective and Reebless. 
Moreover if $\pi_2(M)$ is trivial, then the universal cover of $M$ is homeomorphic to $\R^3$. 
\end{lemma}

\begin{proof}
Let $\F$ be a codimension one non-wandering foliation on a connected closed $3$-manifold $M$. 
By non-wandering property, 
there are no Reeb components. 
By the $C^0$ Novikov Compact Leaf theorem \cite{So}, 
there are no vanishing cycles. 
By Th3.4.VIII \cite{HH}, 
we have that 
$\F$ is $\pi_1$-injective. 
Suppose that 
$\pi_2(M)$ is trivial. 
By Corollary 2.4 \cite{P}, 
the universal cover of $M$ is homeomorphic to $\R^3$. 
\end{proof}

This implies the non-existence 
of  codimension one non-wandering foliations 
on homological spheres. 

\begin{corollary}
There are no codimension one non-wandering foliations on 
homology $3$-spheres.
\end{corollary}

From now on, 
we consider $C^2$ foliations on closed $3$-manifolds.

%
%
%
%
%

\begin{theorem}\label{prop:24}
Let $\F$ be a codimension one  $C^2$ foliation on a closed connected $3$-manifold $M$.  
Suppose there are no leaves of $\F$ whose ends are uncountable.  
Then the following are equivalent: 
\\
1) 
$\F$ is $R$-closed. 
\\
2)
$\F$ is either minimal or compact. 
\\
3) 
$\F$ is non-wandering such that 
either $\mathrm{Cl} = \emptyset$ or $\mathrm{LD} = \emptyset$. 
\end{theorem}

\begin{proof} 
By Lemma \ref{lem12}, 
we have that 
1) and 2) are equivalent. 
Suppose that 
$\F$ is minimal or compact. 
Then $\F$ is non-wandering 
and has either no compact leaves or no locally dense leaves. 
Conversely, 
suppose that 
$\F$ is non-wandering such that 
either $\mathrm{Cl} = \emptyset$ or $\mathrm{LD} = \emptyset$. 
%
By the Duminy theorem for ends \cite{CC2}, 
there are no exceptional leaves. 
If there are no compact leaves, 
then the minimal set is the whole manifold $M$ 
and so $\F$ is minimal. 
Thus we may assume that there are no locally dense leaves. 
Then $M = \mathrm{Cl} \sqcup \mathrm{P}$. 
Assume   $\F$ is not compact. 
Since the union of compact leaves are closed, 
the union $\mathrm{P} = M - \mathrm{Cl}$ of non-compact leaves are nonempty open 
and consists of non-compact proper leaves. 
This contradicts to non-wandering property. 
Thus $\F$ is compact. 
\end{proof}

Note that a codimension one minimal foliation on a closed $3$-manifold 
need not have trivial holonomy.  
In fact, there is a codimension one real-analytic minimal foliation on a closed $3$-manifold 
with non-trivial holonomy 
such that all leaves are diffeomorphic to each other (e.g. Example 3.2 \cite{YT}). 
The polynomial growth of the fundamental group of a manifold implies the following statement. 

\begin{theorem}
Let $\F$ be a codimension one transversely orientable 
$C^2$ foliation on a connected closed $3$-manifold $M$ whose 
fundamental group $\pi_1(M)$ has polynomial growth.
Then the following are equivalent: 
\\
1) 
$\F$ is minimal $($resp. compact$)$. 
\\
2)
$\F$ is non-wandering such that 
$\mathrm{Cl} = \emptyset$ $($resp. $\mathrm{LD} = \emptyset )$. 
\end{theorem}

\begin{proof}
By definitions, we have 
$M = \mathrm{Cl} \sqcup \mathrm{P} \sqcup \mathrm{LD} \sqcup \mathrm{E}$. 
Obviously, the condition 1) implies the condition 2). 
Conversely, 
suppose that 
$\F$ is non-wandering such that 
$\mathrm{Cl} = \emptyset$ (resp. $\mathrm{LD} = \emptyset$). 
The non-wandering property implies that 
there are no Reeb components and so each leaf is $\pi_1$-injective. 
Since the fundamental group of $M$ has polynomial growth, 
so is one of each leaf. 
Since each leaf is a surface and the fundamental group has polynomial growth,  
it has at most two ends. 
By the Duminy theorem for ends \cite{CC2}, 
there are no exceptional leaves. 
Then $M = \mathrm{P} \sqcup \mathrm{LD}$ 
(resp.  $M = \mathrm{Cl} \sqcup \mathrm{P}$) 
and each minimal set consists of locally dense leaves (resp. a compact leaf). 
Since the closure of a non-closed proper leaf contains no locally dense leaves 
(resp. $\mathrm{Cl}$ is closed), 
we have $M = \mathrm{LD}$ (resp. $M = \mathrm{Cl}$).  
This implies that 
$\F$ is minimal (resp. compact). 
\end{proof}

\subsection{On trivial holonomies for non-wandering foliations}

By a Novikov's result \cite{N}, 
the existence of vanishing cycles in a closed $3$-manifold 
implies the existence of Reeb components. 
Theorem \ref{prop:24} implies the following statement. 

\begin{lemma}\label{prop3-6a}
Let $\F$ be a codimension one transversely orientable 
$C^2$ foliation on a closed $3$-manifold $M$.  
Suppose that the fundamental groups of all leaves have the same polynomial growth. 
Then the following are equivalent: 
\\
1) 
$\F$ is without holonomy. 
\\
2)
$\F$ is non-wandering   
such that either $\mathrm{Cl} = \emptyset$ or $\mathrm{LD} = \emptyset$. 
\end{lemma}

\begin{proof}
Suppose that $\F$ is without holonomy. 
By Theorem 1 \cite{T}, 
the foliation $\F$ is either minimal or compact 
and so the assertion holds. 
Conversely, 
suppose that 
$\F$ is non-wandering  
 such that 
either $\mathrm{Cl} = \emptyset$ or $\mathrm{LD} = \emptyset$. 
Since the fundamental group of each leaf $L$ of $\F$ has polynomial growth, 
the surface $L$ has at most two punctured and so the end of $L$ is countable. 
By Theorem \ref{prop:24}, 
the foliation $\F$ is either minimal or compact. 
By Theorem 6 \cite{Sa}, 
all leaves are diffeomorphic to each other.   
Since each compact codimension one transversely orientable foliation on a closed manifold has 
trivial holonomy, 
we may assume that $\F$ is minimal. 
By Theorem 5.1 \cite{N}, 
each leaf is $\pi_1$-injective. 
Since the union of the leaves without holonomy is dense G$_{\delta}$ \cite{EMT}, 
there is a leaf $ L$ of $\mathcal{F}$ without holonomy. 
Proposition 5.1 \cite{Y4} implies that 
$\F$ is without holonomy. 
\end{proof}

It's well known that 
a surface without boundaries whose fundamental group has polynomial growth is 
either an open disk, a sphere, a real projective plane, 
an open annulus, a M\"obius band, 
a torus, or a Klein bottle (cf. \cite{Mi}). 
Moreover 
the polynomial growths of the fundamental groups of 
an open disk, a sphere,  and a real projective plane are zero, 
those of an annulus and a M\"obius band are one, and 
those of a torus and a Klein bottle are two. 
The facts implies the negative answer of Question \ref{q01} 
in the three dimensional case. 

\begin{theorem}\label{prop3-6}
Let $\F$ be a codimension one transversely orientable 
$C^2$ foliation on a connected closed $3$-manifold $M$. 
Then the following are equivalent: 
\\
1) 
The fundamental groups of all leaves have the same polynomial growth. 
\\
2)
$\F$ is without holonomy and has a leaf whose fundamental group has polynomial growth. 
\end{theorem}

\begin{proof} 
Suppose that 
$\F$ is without holonomy and has a leaf whose fundamental group has polynomial growth. 
By Theorem 6 \cite{Sa}, all leaves are diffeomorphic to each other 
and so the fundamental groups of all leaves have the same polynomial growth.   
%
Conversely, 
suppose that the fundamental groups of all leaves have the same polynomial growth. 
Let $k$ be the polynomial growths of the fundamental groups of leaves. 
By Theorem 1 and 2 \cite{Mi}, the polynomial growth $k$ is either zero, one, or two. 
Suppose that $k = 0$.  
The transversely orientability of $\F$ implies that $\F$ has no finite holonomy 
and so $\F$ is without holonomy. 
Suppose that $k = 1$.  
Then each leaf is either an annulus and a M\"obius band. 
By the Duminy theorem for ends \cite{CC2}, there are no exceptional leaves. 
This implies each minimal set is locally dense.  
Since the closure of a non-closed proper leaf contains no locally dense leaves 
we have that $M = \mathrm{LD}$ and so $\F$ is non-wandering. 
Lemma \ref{prop3-6a} implies that $\F$ is without holonomy.  
Thus we may assume that $k = 2$. 
Then each leaf is compact and so $\F$ is a compact foliation. 
Since a compact codimension one transversely orientable foliation is without holonomy, so is $\F$. 
\end{proof}

\section{Examples}

The following example shows that 
the countable condition in Theorem \ref{prop:24} is necessary 
and that 
pointwise almost periodicity  does not correspond to  recurrence. 

\begin{example}\label{prop2.3}
There is a smooth codimension one recurrent  foliation $\F$ 
without compact leaves on a closed $3$-manifold $\Sigma_4 \times \S^1$ 
which is not pointwise almost periodic 
such that $\F$ consists of exceptional leaves and locally dense leaves, 
where $\Sigma_k$ is the genus $k$ closed orientable surface. 
\end{example}

\begin{proof}
Let 
$G$ be the group generated by a circle diffeomorphisms $f, g$ 
in \cite{Sa} with a unique Cantor minimal set 
$\mathcal{M}$
and  
$f_1, f_2: (1/3, 2/3) \to (1/3, 2/3)$ smooth diffeomorphisms such that 
each orbit of the group generated by $f_1, f_2$ is dense. 
Note that 
$(1/3, 2/3)$ is a connected component of $\S^1 - \mathcal{M}$. 
We can choose $f_1, f_2$ such that 
the extensions of $f_i$ are 
circle smooth diffeomorphism $F_i: \S^1 \to \S^1$ 
whose supports are $(1/3, 2/3) \subset \S^1$, 
where $\S^1 = \R/\Z$. 
Consider the product foliation $\{ \Sigma_4 \times \{ x \} \mid x \in \S^1 \}$ 
and four disjoint loops $\gamma_f, \gamma_g, \gamma_1, \gamma_2$ 
in $\Sigma_4$ such that $\Sigma_4 - \sqcup_{i \in \{ f, g,1,2 \} } \gamma_i$ is a punctured disk. 
Taking holonomy maps $\mathrm{id} \times F_i: \gamma_i \times \S^1 \to \gamma_i \times \S^1$ 
for a circle bundle over $\Sigma$, 
we obtain 
a codimension one foliation $\F$ 
such that each leaf is exceptional or locally dense. 
Therefore $\F$ is not pointwise almost periodic but recurrent. 
\end{proof}

The following example shows that 
recurrence does not correspond to non-wandering property. 

\begin{example}\label{prop2.4}
There is a smooth codimension one non-wandering foliation $\F$  on $\Sigma_3 \times \S^1$  
which is not recurrent. 
\end{example}

\begin{proof}
%
%

Consider a minimal $\Z^2$-action generated by two 
orientation-preserving diffeomorphisms $f$ and $g$ on $(0, 1)$ 
(e.g. diffeomorphisms which are conjugate to translations on $\R$) 
and extend 
$f, g: \R \to \R$ into diffeomorphisms such that  
$f(x + 1) = f(x) + 1$ 
and 
$g(x + 1) = g(x) + 1$. 
Define $h: \R \to \R$ by $h(x) = x + 1$. 
Extend $f, g, h$ into diffeomorphisms by adding common fixed point 
$\infty$.  
The resulting underlying manifold $\S^1_{\infty} = \R \sqcup \{ \infty \}$ 
is the one-point compactification of $\R$. 
Then the resulting foliated bundle $(M, \F)$ is desired, 
where $M := \Sigma_3 \times \S^1_{\infty}$.  
Indeed, 
the leaf class space $M /\hat{\F}$ consists of 
three points $\hat{L}_{\infty}, \hat{L}_{\mathrm{op}}, \hat{L}_{\Z}$ 
such that 
$\hat{L}_{\infty}$ is a closed point, 
$\hat{L}_{\mathrm{op}}$ is an open point with $\overline{\hat{L}_{\mathrm{op}}} = M /\hat{\F}$, 
and 
$\hat{L}_{\Z}$ is neither a closed point nor an open point. 
\end{proof}

Notice that 
there are codimension one Reebless foliation on $\T^3$ 
which is wandering. 
Indeed, 
consider a non-trivial translations $f$. 
Adding the ideal point to $\R$ as a fixed point, 
we obtain a smooth homeomorphism $\widetilde{f}$
on the circle $\S^1_{\infty}$. 
Consider the suspension foliation $\F_1$ on a torus $\T^2$ 
and a product foliation $\F_2 := \{ L \times \S^1 \mid L \in \F_1 \}$ on a torus $\T^3$. 
Then the resulting  foliation is wandering but Reebless.

\end{document}